\documentclass{svjour3}
\usepackage{amssymb}
\usepackage{bm} 
\usepackage{graphicx}
\usepackage{graphics}
\usepackage{psfrag}
\usepackage{amsmath}
\usepackage{amscd}
\usepackage{amsfonts}
\usepackage{float}
\usepackage{latexsym}
\usepackage{xcolor}
\usepackage{lscape}
\usepackage{mathrsfs}

\newcommand{\ds}{\displaystyle}
\newcommand{\f}{\frac}

\newcommand{\om}{\Omega}
\newcommand{\p}{\partial}

\newcommand{\xx}{{\bm x}}

\begin{document}
\title{Optimal order finite element approximations for variable-order time-fractional diffusion equations 
}
\authorrunning{Zheng, Zeng and Wang}
\titlerunning{FEM for variable-order FDEs}        

\author{Xiangcheng Zheng  \and
  Fanhai Zeng \and Hong Wang}

\institute{Xiangcheng Zheng \at
              Department of Mathematics, University of South Carolina, Columbia, South Carolina, 29208 USA\\
              \email{xz3@math.sc.edu}  
              \and
         Fanhai Zeng \at Department of Mathematics, National University of Singapore, Singapore 119076\\
              \email{fanhaiz@foxmail.com}     
           \and
         Hong Wang \at  Department of Mathematics, University of South Carolina, Columbia, South Carolina, 29208 USA\\
              \email{hwang@math.sc.edu} 
}

\date{Received: date / Accepted: date}

\maketitle

\begin{abstract}
We study a fully discrete finite element method for variable-order time-fractional diffusion equations with a time-dependent variable order. Optimal convergence estimates are proved with the first-order accuracy in time (and second order accuracy in space) under the uniform or graded temporal mesh without full regularity assumptions of the solutions. Numerical experiments are presented to substantiate the analysis.
\end{abstract}

\section{Introduction}
Fractional partial differential equations (FPDEs) have shown to provide adequate descriptions for the challenging phenomena such as the anomalous diffusive transport and the memory effect \cite{MeeSik,MetKla00,Pod}. For instance, in the diffusive transport in heterogeneous porous media, a large amount of particles may get absorbed to the surface of the rock. Thus, the travel time of the adsorbed particles may deviate from that of the particles in the bulk phase \cite{ZhoStr}, which leads to a subdiffusive transport that can be modeled by a time-fractional diffusion PDE (tFPDE) \cite{MeeSik,MetKla00}. 

Extensive mathematical and numerical analysis of FPDEs has been conducted  \cite{Die,DieFor,ErvRoo05,KilSri,LeMclMus,Luc,Pod,SakYam,StyOriGra}, and it is gradually getting clear that the FPDEs introduce mathematical issues that are not common in the context of integer-order PDEs. For instance, the smoothness of the coefficients and right-hand side of a linear elliptic or parabolic fractional PDE in one space dimension cannot ensure the smoothness of its solutions \cite{JinLazPas,StyOriGra,WanYanZhu,WanZha}. Hence, many error estimates in the literature that were proved under full regularity assumptions of the true solutions are inappropriate. 

Variable-order tFPDEs, in which the order of the fractional derivatives varies in time as $t \rightarrow 0$ to accommodate the impact of the local initial condition at time $t=0$, should be a natural candidate to eliminate the nonphysical singularity of the solutions to (constant-order) tFPDEs and open up opportunities for modeling multiphysics phenomena from nonlocal to local dynamics and vice versa \cite{LiWan,PanPer,SunCheChe,ZenZhaKar,ZhuLiu}. 

Due to the difficulties of solving variable-order tFPDEs analytically, several numerical methods have been developed (see e.g. \cite{ZhuLiu,ZenZhaKar}) under certain smoothness assumptions of the solutions. It was shown in \cite{StyOriGra} that the first order time derivatives of solutions to the $\alpha$-order time-fractional diffusion equations (tFDEs) exhibit the singularity of $\mathcal O(t^{\alpha-1})$ at the initial time $t=0$, which leads to a sub-optimal convergence of the fully discrete finite difference method. It was also proved that by using the graded temporal mesh with a proper chosen mesh grading parameter according to the singularity of the solutions, the optimal convergence rate of the proposed finite difference method can be recovered.

 Recently, the wellposedness of a variable-order tFDE model and the regularity of its solutions were studied in \cite{WanZheJMAA}. In particular, the solutions have full regularity like those to the integer-order tFDEs if the variable order has an integer limit at $t = 0$ or exhibit singularity at $t = 0$ like in the case of the constant-order tFDEs if the variable order has a non-integer
value at time $t = 0$. Based on these theoretical results, we present a first order time-discretized finite element method for this variable-order tFDE model. When the variable order smoothly transit the fractional order model to the integer order ones near the initial time, the solutions have full regularity and the optimal convergence is proved under the uniform temporal mesh. Otherwise, a graded mesh with a properly chosen mesh grading parameter in terms of the singularity of the solutions at the initial time is applied to recover the optimal convergence rate.

 The rest of the papers are organized as follows: In \S 2 a variable-order tFDE model and the auxiliary results to be used subsequently were presented. In \S 3 a first order time-discretized finite element method was developed for the proposed model and we proved the corresponding optimal error estimates under the uniform or graded temporal mesh in terms of the regularity of the solutions in \S 4. Several numerical experiments were presented in \S 5 to demonstrate the theoretical analysis.

\section{Model problem and preliminaries} 
Let $m \in \mathbb{N}$, $1 \le p \le \infty$, $\Omega \subset \mathbb{R}^d$ ($d=1,2,3$) be a simply-connected bounded domain with smooth boundary $\p \om$ and ${\cal I}\subset [0,\infty)$ be a bounded interval. Let $L_p(\om)$ be the spaces of the $p$-th power Lebesgue integrable functions on $\om$ and $W^m_p(\om)$ be the Sobolev spaces of functions with derivatives of order up to $m$ in $L_p(\om)$. Let $H^m(\om) := W^m_2(\om)$ and $H^m_0(\Omega)$ be the completion of $C^\infty_0(\om)$, the space of infinitely many time differentiable functions with compact support in $\om$, in $H^m(\om)$ \cite{AdaFou}. For the case of non-integer order $s$, the fractional Sobolev spaces $H^s(\Omega)$ are defined by interpolation, see \cite{AdaFou}. Furthermore, for the Banach space ${\cal X}$, we introduce the Sobolev spaces involving time \cite{AdaFou,Eva}
$$W^m_p({\cal I};{\cal X}) := \big \{ f : f_t^{(l)}(\cdot,t) \in {\cal X},~t \in {\cal I},~ \bigl \|  f_t^{(l)}(\cdot,t) \bigr \|_{\cal X}  \in L_p({\cal I}), ~ 0 \le l \le m \big \}.$$
In particular, $W^0_p({\cal I};{\cal X}) = L_p({\cal I};{\cal X})$ for $1 \le p \le \infty$. We also let $C^m({\cal I};{\cal X})$ be the spaces of functions with continuous derivatives up to order $m$ on ${\cal I}$ equipped with the norm
$$\|f\|_{C^{m}(\cal I;\cal X)}:=\max_{0\leq l\leq m}\sup_{t\in \cal I}\,\bigl \|  f_t^{(l)}(\cdot,t) \bigr \|_{\cal X}.$$

In this paper we study the initial-boundary value problem of a variable-order linear tFDE 
\begin{equation}\label{ModelR}\begin{array}{c}
u_t + k(t)~{}_0^{R}D_t^{1-\alpha(t)} u +\mathcal{L} u= f(\bm x,t),~~(\bm x,t) \in \Omega\times(0,T]; \\ [0.05in]
\ds u(\bm x,0)=u_0(\bm x),~\bm x\in \Omega; \quad u(\bm x,t) = 0,~(\bm x,t) \in \p \Omega\times[0,T]. 
\end{array}\end{equation}
Here $u_t$ refers to the first-order partial derivative in time, $\xx := (x_1,\cdots,x_d)$, $\mathcal{L}:=-\nabla \cdot \big(\bm K(\bm x)\nabla\big)$ with $\nabla := (\p/\p x_1,\cdots,\p/\p x_d\big)^T$ and ${\bm K}(\xx):=(k_{ij}(\bm x))_{i,j=1}^d$ the diffusion tensor. We make the following assumptions throughout the paper.
\paragraph{Assumption A.}
\begin{eqnarray}
& \alpha, k \in C[0,T],~~k_m:=\min_{t\in [0,T]}k(t)>0; \notag\\[0.05in] \label{alpha}
& 0< \alpha_m := \inf_{t \in [0,T]} \alpha(t)\leq \alpha(t)  \le 1, ~~t \in [0,T],\\[0.05in]
&\ds\lim_{t \rightarrow 0^+} \big ( \alpha(t) - \alpha(0) \big) \ln t = 0; \notag\\[0.025in] 
\notag
& 0 < K_m \le {\bm \xi}^T\! \bm K {\bm \xi} \le K_M < \infty,~~ \forall {\bm \xi} \in \mathbb{R}^d, ~~ | {\bm \xi} | = 1, \\[0.05in]
&\ds k_{ij}\in C^1(\Omega),~~1\leq i,j\leq d.\notag
\end{eqnarray}
The variable-order Riemann-Liouville fractional derivative is defined by \cite{ZenZhaKar,ZhuLiu} 
$$\ds {}_a^RD_t^{1-\alpha(t)}g(t) := \bigg [ \frac{1}{\Gamma\big(\alpha(t)\big)}\frac{d}{d\xi} \int_a^\xi\frac{g(s)}{(\xi-s)^{1-\alpha(t)}}ds\bigg]\bigg|_{\xi=t}.$$
Moreover, we also use the variable-order fractional integral operator ${}_aI_t^{\alpha(t)}$ and the Caputo fractional differential operator ${}_a^CD_t^{\alpha(t)}$ \cite{ZenZhaKar,ZhuLiu}
$$\ds {}_aI_t^{\alpha(t)}g(t):=\frac{1}{\Gamma(\alpha(t))}\int_a^t\frac{g(s)}{(t-s)^{1-\alpha(t)}}ds, \quad
{}_a^CD_t^{1-\alpha(t)}g(t)={}_aI_t^{\alpha(t)}g'(t).$$
\begin{remark} The constant-order analogue of the proposed model is known as the mobile-immobile time-fractional diffusion equations, see e.g., \cite{LiuLiu}. 
\end{remark}

The relation between the Riemann-Liouville and Caputo fractional derivatives \cite{Die,ErvRoo05,Pod} was extended to the variable-order analogues \cite{ZhuLiu}.
\begin{lemma}\label{lem:R-C}
	Let $g \in W^1_1(0,T)$. Then \vspace{-0.1in}
	\begin{equation*}
	{}_0^RD_t^{1-\alpha(t)}g(t) = {}_0^CD_t^{1-\alpha(t)}g(t) + \frac{g(0) t^{\alpha(t)-1}}{\Gamma(\alpha(t))},~~t\in (0,T].
	\end{equation*}
\end{lemma}

In this paper we use $Q$ to denote generic positive constants that may assume different values at different occurrences. For convenience, we may drop the subscript $L_2$ in $(\cdot,\cdot)_{L_2}$ and $\| \cdot \|_{L_2}$ as well as the notation $\om$ in the Sobolev spaces and norms, and abbreviate $W^m_p(0,T;{\cal X})$ and $W^m_p({\cal X})$, when no confusion occurs.

\section{Fully discrete finite element method for variable-order tFDEs}
By Lemma \ref{lem:R-C} and Theorem \ref{thm:Wellpose}, the Riemann-Liouville variable-order tFDE (\ref{ModelR}) and the following Caputo variable-order tFDE
\begin{equation}\label{ModelC}
\f{\p u}{\p t} + k(t)~{}_0^{C}D_t^{1-\alpha(t)} u+\mathcal{L} u= -\frac{k(t)u_0(\bm x) t^{\alpha(t)-1}}{\Gamma(\alpha(t))}+ f(\bm x,t)
\end{equation}
 coincide. So we will develop and analyze the corresponding finite element schemes for (\ref{ModelC}).

Let $t_n := T(n/N)^r$, $0\leq n\leq N$ be a partition of $[0,T]$, which forms a graded mesh when $r>1$ and reduces to a uniform partition for $r=1$. Applying the mean-value theorem we bound $\tau_n := t_n - t_{n-1}$ by 
  \begin{equation}\label{bndtau}
rT \frac{ (n-1)^{r-1}}{N^r}\leq \tau_n=T\Big(\frac{n}{N}\Big)^r-T\Big(\frac{n-1}{N}\Big)^r \leq rT \frac{ n^{r-1}}{N^r}, \quad r \ge 1,\quad 1\leq n\leq N.
  \end{equation}

  Define $\Omega_e$ a quasi-uniform partition of $\Omega$ with parameter $h$ and $S_h$ the space of piece-wise linear functions on $\Omega$ with compact support. The Ritz projection $\Pi:H^1_0(\Omega)\rightarrow S_h$ defined by
 \begin{equation}\label{Ritz}
 \ds \big(\bm K(\cdot)\nabla(g-\Pi g),\nabla \chi\big)=0,~~\forall \chi\in S_h,~\mbox{for}~g\in H_0^1(\Omega),
 \end{equation}
 has the following approximation property \cite{Tho}
 \begin{equation}\label{ritz}
 \ds \|g-\Pi g\|_{L_2(\Omega)}\leq Qh^2\|g\|_{H^2(\Omega)},~~\forall ~g\in H^2(\Omega)\cap H_0^1(\Omega).
 \end{equation}

We discretize $u_t$ and ${}_0^CD_t^{1-\alpha(t)}u$ at $t=t_n$, $1\leq n\leq N$ by
 \begin{equation}\label{Disc:e1}\begin{array}{rl}
\ds u_t(\bm x,t_n)&\ds = \delta_{\tau_n} u(\bm x,t_n) + r_{1,n}\\[0.1in]
& \ds\hspace{-0.15in}:= \frac{u(\bm x,t_n)-u(\bm x,t_{n-1})}{\tau_n} + \f1{\tau_n}\int_{t_{n-1}}^{t_n} u_{tt}(\bm x,t)(t-t_{n-1})dt, \\[0.15in]
\ds {}_0^CD_t^{1-\alpha(t_n)}&\ds u(\bm x,t_n)  \\
 &\ds\hspace{-0.25in}= \frac{1}{\Gamma(\alpha(t_n))} \sum_{k=1}^n \Big [\int_{t_{k-1}}^{t_k} \f{\delta_{\tau_k}u(\bm x,t_k)}{(t_n-t)^{1-\alpha(t_n)}}dt  + \int_{t_{k-1}}^{t_k} \frac{u_t - \delta_{\tau_k}u(\bm x,t_k)}{(t_n-t)^{1-\alpha(t_n)}} dt \Big ] \\[0.175in]
& \ds\hspace{-0.25in} =: \delta_{\tau}^{1-\alpha(t_n)} u(\bm x,t_n) + r_{2,n}, \\
 \end{array}\end{equation}
 with
 \begin{equation}\label{Disc:e2}\begin{array}{l}
 \ds \delta_{\tau}^{1-\alpha(t_n)} u(\bm x,t_n) \\[0.1in]
 \ds\quad \quad:= \frac{1}{\Gamma(1+\alpha(t_n))} \sum_{k=1}^n \Big [\big ( t_n - t_{k-1} \big)^{\alpha(t_n)} - \big(t_n-t_k\big)^{\alpha(t_n)} \Big] \delta_{\tau_k} u(\bm x,t_k) \\[0.15in]
 \ds \quad \quad= \frac{1}{\Gamma(1+\alpha(t_n))}\sum_{k=1}^n b^n_k \big (u(\bm x,t_k) - u(\bm x,t_{k-1}) \big ),\\[0.15in]
 r_{2,n}\ds  := \frac{1}{\Gamma(\alpha(t_n))} \sum_{k=1}^n \int_{t_{k-1}}^{t_k} \frac{u_t - \delta_{\tau_k}u(\bm x,t_k)}{(t_n-t)^{1-\alpha(t_n)}} dt \\[0.15in]
 \ds \qquad = \frac{1}{\Gamma(\alpha(t_n))} \sum_{k=1}^n \int_{t_{k-1}}^{t_k} \frac{1}{\tau_k (t_n-t)^{1-\alpha(t_n)}}\Big [\int_{t_{k-1}}^{t_k}\int_s^t u_{tt}(\bm x,\theta) d\theta ds \Big ] dt ,
 \end{array}\end{equation}
 where 
 $$b^n_k :=\frac{ (t_n - t_{k-1})^{\alpha(t_n)} - (t_n-t_k)^{\alpha(t_n)}}{\tau_k},~~1 \le k \le n \le N$$ 
 has the following properties \cite{StyOriGra}
 \begin{equation}\label{bn}
 \left\{\begin{array}{l}
 \tau_n^{\alpha(t_n)-1} = b^n_n > b^n_{n-1} > \cdots > b^n_k > \ldots b^n_1 > 0, \\[0.05in]
 \ds \alpha(t_n)(t_n-t_{k-1})^{\alpha(t_n)-1}\leq b^n_k\leq \alpha(t_n)(t_n-t_k)^{\alpha(t_n)-1}.
 \end{array} \right.
 \end{equation}
 Let $u^n:=u(\bm x,t_n)$. We plug (\ref{Disc:e1}) into (\ref{ModelC}), multiply $\chi\in H_0^1(\Omega)$ on both sides and integrate the resulting equation on $\Omega$ to get the weak formulation of (\ref{ModelC})
 \begin{equation}\label{ref}
 \begin{array}{l}
 \ds\hspace{-0.1in} (\delta_{\tau_n}  u^n,\chi)+\big(\bm K(\cdot)\nabla u^n,\nabla \chi\big)= -k(t_n)(\delta_{\tau}^{1-\alpha(t_n)}  u^n,\chi)-\frac{k(t_n)t_n^{\alpha(t_n)-1}}{\Gamma(\alpha(t_n))}(u_0,\chi)\\[0.15in] 
 \quad\quad\ds+(f(\cdot,t_n),\chi)-\big(k(t_n) r_{2,n}+r_{1,n},\chi\big),~~\chi\in H_0^1(\Omega),~~n=1,\cdots,N.
 \end{array}
 \end{equation}
 We drop the truncation error terms to obtain a first order time-discretized finite element scheme for (\ref{ModelC}): find $u_h^n\in S_h$ such that
  \begin{equation}\label{FEM}
  \begin{array}{l}
  \ds \hspace{-0.18in}(\delta_{\tau_n}  u_h^n,\chi)+\big(\bm K(\cdot)\nabla u_h^n,\nabla \chi\big)= -k(t_n)(\delta_{\tau}^{1-\alpha(t_n)}  u_h^n,\chi)-\frac{k(t_n)t_n^{\alpha(t_n)-1}}{\Gamma(\alpha(t_n))}(u_0,\chi)\\[0.15in] 
  \quad\quad\ds+(f(\cdot,t_n),\chi),~~\chi\in S_h,~~n=1,\cdots,N.
  \end{array}
  \end{equation}

 \section{Convergence estimates of the finite element approximations}
 We prove the optimal error estimates of the finite element approximations under the uniform or graded mesh in terms of the regularity of the solutions to the proposed model.
\subsection{Analysis of truncation errors}
 The estimates of the local truncation errors $r^n_1$ and $r^n_2$ in (\ref{Disc:e1}) and (\ref{Disc:e2}) are given in the following theorem.
 \begin{theorem}\label{thm:Local} Suppose that $u_0\in \check H^4$ and $f\in H^1(0,T;\check{H}^2)\cap H^2(0,T;L_2)$. For the case of $\alpha(0)=1$, $\alpha'(0)=0$ and $\lim_{t\rightarrow 0^+}\alpha'(t)\ln t$ is finite, the following estimate holds under the uniform temporal partition 
 	\begin{equation}\label{thmLocal:e1}
 	\begin{array}{l}
 	\ds\hspace{-0.1in} \|r_{1}\|_{\hat L_\infty(0,T;L_2)}:=\max_{1\leq n\leq N}\|r_{1,n}\|_{L_2}\leq Q  Q_0N^{-1},~~\| r_{2}\|_{\hat L_\infty(0,T;L_2)} \leq Q  Q_0N^{-1},\\[0.1in]
 	 \ds \quad\quad Q_0:=\big(\|u_0\|_{\check{H}^4}+\|f\|_{H^1(0,T;\check{H}^2)}+\|f\|_{H^2(0,T;L_2)}\big).
 	\end{array}
 	\end{equation} 
 Otherwise, the following estimates hold under the graded mesh with $r\alpha(0)>1$
 	\begin{equation}\label{thmLocal:e2}\begin{array}{l}
 	\ds  \|r_{1,n} \|_{\hat L_\infty(0,T;L_2)}\leq Q  Q_0n^{-1-r(1-\alpha(0))}N^{r(1-\alpha(0))}, \\[0.1in]
 	\ds \|r_{2,n}\|_{\hat L_\infty(0,T;L_2)} \le Q  Q_0\Big( \hat\delta_{n,1}N^{r(1-\alpha(0)-\alpha(1))}\\[0.1in]
 	\ds\qquad~~\qquad\qquad\qquad +(1-\hat\delta_{n,1}) N ^{r(1-\alpha(0) - \alpha(t_n))} n^{-1-r(1-\alpha(0) - \alpha(t_n))} \Big ).
 	\end{array}\end{equation}
 	Here $\hat\delta_{m,n}$ is the Kronecker delta function.
 \end{theorem}
 \begin{proof}
 	The proof is exactly the same as that of Theorem 7 in \cite{WanZhe} with $\alpha_*$ in that theorem replaced by $\alpha(0)$ according to Theorem \ref{thm:u2}. 
 	\end{proof}
 	
We bound another two truncation terms for $1\leq n\leq N$
 \begin{equation}\label{r3r4}
 \ds r_{3,n}:=\delta_{\tau_n} \big(u(\bm x,t_n)- \Pi u(\bm x,t_n)\big),~~r_{4,n}:=\delta_{\tau}^{1-\alpha(t_n)}\big( u(\bm x,t_n)- \Pi u(\bm x,t_n)\big),
 \end{equation}
 for the convenience of the convergence estimates.
 \begin{theorem}\label{thmtrun}
 Suppose $u_0\in \check H^4$, $f\in H^1(0,T;\check H^2)$ and the Assumption A holds. Then the following estimates hold under the uniform temporal partition for the case of $\alpha(0)=1$, $\alpha'(0)=0$ and $\lim_{t\rightarrow 0^+}\alpha'(t)\ln t$ is finite 
 \begin{equation}\label{thmtrun:e1}
 \ds \|r_{3,n}\|_{\hat L_\infty(0,T;L_2)}+\|r_{4,n}\|_{\hat L_\infty(0,T;L_2)}\leq QQ_1h^2,~~Q_1:=\|u_0\|_{\check{H}^4}+\|f\|_{H^1(0,T;\check{H}^2)}
 \end{equation}
 and under the graded mesh otherwise
 \begin{equation}\label{thmtrun:e2}
  \begin{array}{l}
\ds\|r_{3,n}\|_{\hat L_\infty(0,T;L_2)}\leq QQ_1h^2\big(\delta_{n,1}+(1-\delta_{n,1})n^{r(\alpha(0)-1)}\big)N^{-r(\alpha(0)-1)},\\[0.15in]
\ds\|r_{4,n}\|_{\hat L_\infty(0,T;L_2)}\\[0.15in]
\ds\qquad\leq QQ_1h^2\big(\delta_{n,1}+(1-\delta_{n,1})n^{r(\alpha(0)+\alpha(t_n)-1)}\big)N^{-r(\alpha(0)+\alpha(t_n)-1)}.
\end{array}
 \end{equation}
 \end{theorem}
 \begin{proof}
 	When $\alpha(0)=1$, $\alpha'(0)=0$ and $\lim_{t\rightarrow 0^+}\alpha'(t)\ln t$ is finite, $u\in C^1\big([0,T];H^2(\Omega)\big)$ by Theorem \ref{thm:Wellpose} so a uniform partition of $[0,T]$ suffices. We apply (\ref{ritz}) to obtain
 	\begin{equation}
 	\begin{array}{rl}
 	   \ds \|r_{3,n}\|&\ds=\tau_n^{-1}\|(I-\Pi)(u^n-u^{n-1})\|\\[0.1in]
 	   &\ds\leq Qh^2\tau_n^{-1}\|u^n-u^{n-1}\|_{H^2}\leq Qh^2\|u\|_{W^1_\infty(0,T;H^2)}
 	\end{array}
 	\end{equation}
 	and
 	\begin{equation}\label{trun:e1}
 	\begin{array}{l}
 	\ds\|r_{4,n}\|=\frac{1}{\Gamma(1+\alpha(t_n))}\bigg\|\sum_{k=1}^{n}b^n_k(I-\Pi)(u^k-u^{k-1})\bigg\|\\[0.15in]
 	\ds~~~~~\,~~\leq Qh^2\sum_{k=1}^{n}b^n_k\|u^k-u^{k-1}\|_{H^2} \leq Qh^2\|u\|_{W^1_\infty(0,T;H^2)}\sum_{k=1}^{n}b^n_k\tau_k\\[0.15in]
 	\ds~~~\,~~~~\leq Qh^2\|u\|_{W^1_\infty(0,T;H^2)},
 	\end{array}
 	\end{equation}
 	where $I$ refers to the identity operator. Then an application of Theorem \ref{thm:Wellpose} leads to (\ref{thmtrun:e1}).
 	
 	For other cases, we only need to consider the case that $\alpha(0)<1$. The graded mesh with mesh grading $r$ will be used to capture the singularity of the solutions at the initial time. By (\ref{Wellpose:e2}) and the mean-value theorem we bound $r_{3,n}$ by
 		\begin{equation*}
 		\begin{array}{rl}
 		\ds \|r_{3,n}\|&\ds=\frac{1}{\tau_n}\|(I-\Pi)(u^n-u^{n-1})\|\\[0.15in]
 		&\hspace{-0.2in}\ds\leq \frac{Qh^2}{\tau_n}\|u^n-u^{n-1}\|_{H^2}= \frac{Qh^2}{\tau_n}\bigg\|\int_{t_{n-1}}^{t_n}u_tdt\bigg\|_{H^2}\leq \frac{QQ_1h^2}{\tau_n}\int_{t_{n-1}}^{t_n}t^{\alpha(0)-1}dt\\[0.15in]
 		&\hspace{-0.2in}\ds\leq \left\{
 		\begin{array}{l}
 		\ds QQ_1h^2t_1^{\alpha(0)-1}=\frac{QQ_1h^2}{N^{r(\alpha(0)-1)}},~n=1,\\[0.1in]
 		\ds QQ_1h^2t_{n-1}^{\alpha(0)-1}\leq \frac{QQ_1h^2(n-1)^{r(\alpha(0)-1)}}{N^{r(\alpha(0)-1)}},~n>1.
 		\end{array}
 		\right.
 		\end{array}
 		\end{equation*}
  	We remain to bound $r_{4,n}$, which requires a careful argument. From (\ref{trun:e1}) we have
  		\begin{equation*}
  		\begin{array}{rl}
  		\ds \|r_{4,n}\|&\ds=\frac{1}{\Gamma(1+\alpha(t_n))}\bigg\|\sum_{k=1}^{n}b^n_k(I-\Pi)(u^k-u^{k-1})\bigg\|\\[0.15in]
  	&\ds	\leq Qh^2\sum_{k=1}^{n}b^n_k\|u^k-u^{k-1}\|_{H^2} \leq Qh^2\sum_{k=1}^{n}b^n_k\bigg\|\int_{t_{k-1}}^{t_k}u_tdt\bigg\|_{H^2}\\[0.2in]
  		&\ds\leq QQ_1h^2\sum_{k=1}^{n}b^n_k\int_{t_{k-1}}^{t_k}t^{\alpha(0)-1}dt:=QQ_1h^2\sum_{k=1}^{n}J_{n,k}.
  		\end{array}
  		\end{equation*}
 	When $n=1$, $J_{1,1}$ can be bounded by
 	\begin{equation*}
 	J_{1,1}= \frac{b^1_1t_1^{\alpha(0)}}{\alpha(0)}=\frac{t_1^{\alpha(1)+\alpha(0)-1}}{\alpha(0)}\leq QN^{-r(\alpha(1)+\alpha(0)-1)}.
 	\end{equation*}
 	For $n>1$, we first bound $J_{n,1}$ and $J_{n,n}$ by (\ref{bndtau}) and the mean-value theorem
 	\begin{equation}\label{trun:e3}
 	\begin{array}{rl}
 	\ds|J_{n,1}|&\ds=\frac{b^n_1t_1^{\alpha(0)}}{\alpha(0)}=\frac{(t_n^{\alpha(t_n)}-(t_n-t_1)^{\alpha(t_n)})t_1^{\alpha(0)}}{\tau_1\alpha(0)}\\[0.15in]
 	&\ds\leq Q(t_n-t_1)^{\alpha(t_n)-1}t_1^{\alpha(0)}=Q\bigg(\frac{n^r-1}{N^r}\bigg)^{\alpha(t_n)-1}\frac{1}{N^{r\alpha(0)}}\\[0.15in]
 	&\ds\leq Q\frac{n^{r(\alpha(t_n)-1)}}{N^{r(\alpha(0)+\alpha(t_n)-1)}},
 	\end{array}
 	\end{equation}
 	\begin{equation}\label{trun:e4}
 	\begin{array}{l}
 	\ds |J_{n,n}|=\frac{b^n_n(t_n^{\alpha(0)}-t_{n-1}^{\alpha(0)})}{\alpha(0)}\leq Q\tau_n^{\alpha(t_n)-1}t_{n-1}^{\alpha(0)-1}\tau_n\\[0.15in]
 	\ds~~~~~~~\leq Q\frac{n^{r(\alpha(t_n)+\alpha(0)-1)-\alpha(t_n)}}{N^{r(\alpha(t_n)+\alpha(0)-1)}}.
 	\end{array}
 	\end{equation}
 	We remain to consider the case $n\geq 3$ since the estimates (\ref{trun:e3}) and (\ref{trun:e4}) have covered the case $n=2$. By (\ref{bndtau}) and the mean-value theorem we obtain
 	\begin{equation*}
 	\begin{array}{rl}
 	\ds \sum_{k=\lceil n/2\rceil+1 }^{n-1}J_{n,k}&\ds=\sum_{k=\lceil n/2\rceil+1 }^{n-1}\frac{\alpha(t_n)}{\tau_k}\int_{t_{k-1}}^{t_k}\frac{dt}{(t_n-t)^{1-\alpha(t_n)}}\int_{t_{k-1}}^{t_k}t^{\alpha(0)-1}dt\\[0.2in]
 	&\ds=\frac{\alpha(t_n)}{\alpha(0)}\sum_{k=\lceil n/2\rceil+1 }^{n-1}\frac{t_k^{\alpha(0)}-t_{k-1}^{\alpha(0)}}{\tau_k}\int_{t_{k-1}}^{t_k}\frac{dt}{(t_n-t)^{1-\alpha(t_n)}}\\[0.2in]
 &\ds \leq Q\sum_{k=\lceil n/2\rceil+1 }^{n-1}t_{k-1}^{\alpha(0)-1}\int_{t_{k-1}}^{t_k}\frac{dt}{(t_n-t)^{1-\alpha(t_n)}}\\[0.2in]
 &\ds\leq Qt_{n}^{\alpha(0)-1}\sum_{k=\lceil n/2\rceil+1 }^{n-1}\int_{t_{k-1}}^{t_k}\frac{dt}{(t_n-t)^{1-\alpha(t_n)}}\\[0.2in]
 	&\ds\leq Qt_{n}^{\alpha(0)-1}(t_n-t_{\lceil n/2\rceil})^{\alpha(t_n)}\leq Qt_n^{\alpha(0)+\alpha(t_n)-1}= Q\frac{n^{r(\alpha(0)+\alpha(t_n)-1)}}{N^{r(\alpha(0)+\alpha(t_n)-1)}},
 	\end{array}
 	\end{equation*}
 	and
 	\begin{equation*}
 	\begin{array}{rl}
 	\ds \sum^{\lceil n/2\rceil }_{k=2}J_{n,k}&\ds\leq Q\sum^{\lceil n/2\rceil }_{k=2}t_{k-1}^{\alpha(0)-1}(t_n-t_k)^{\alpha(t_n)-1}\tau_k\leq Qt_n^{\alpha(t_n)-1}\sum^{\lceil n/2\rceil }_{k=2}t_{k-1}^{\alpha(0)-1}\tau_k\\[0.2in]
 	&\ds \leq Q\frac{n^{r(\alpha(t_n)-1)}}{N^{r(\alpha(t_n)-1)}}\sum^{\lceil n/2\rceil }_{k=2}\frac{(k-1)^{r(\alpha(0)-1)}}{N^{r(\alpha(0)-1)}}\frac{k^{r-1}}{N^r}\\[0.2in]
 	&\ds \leq Q\frac{n^{r(\alpha(t_n)-1)}}{N^{r(\alpha(t_n)+\alpha(0)-1)}}\sum^{\lceil n/2\rceil }_{k=2}k^{r\alpha(0)-1}\leq Q \frac{n^{r(\alpha(t_n)+\alpha(0)-1)}}{N^{r(\alpha(t_n)+\alpha(0)-1)}}.
 	\end{array}
 	\end{equation*}
We summarize the above estimates to finish the proof.
  	\end{proof}
\subsection{Convergence estimates of the finite element approximations}
We prove the optimal error estimate of the fully discrete finite element method (\ref{FEM}) by the following theorem. 
\begin{theorem}\label{thm:conv}
Suppose that $u_0\in \check H^4$ and $f\in H^1(0,T;\check{H}^2)\cap H^2(0,T;L_2)$. We set $r=1$ for the case of $\alpha(0)=1$,  $\alpha'(0)=0$ and $\lim_{t\rightarrow 0^+}\alpha'(t)\ln t$ is finite and $r=2/\alpha(0)$ otherwise. Then the following optimal order error estimate holds 
\begin{equation*}
\begin{array}{l}
\ds \| u_h - u\|_{\hat L_\infty(0,T;L_2)}\leq Q\big(\|u_0\|_{\check{H}^4}+\|f\|_{H^1(0,T;\check{H}^2)}+\|f\|_{H^2(0,T;L_2)}\big)\big(N^{-1}+h^2\big).
\end{array}
\end{equation*}
Here $Q=Q(T,\alpha_m,\|k\|_{C[0,T]})$.
\end{theorem}
\begin{proof}
	We split the error by $e^n:=u_h^n-u^n=\xi^n+\eta^n$ where $\xi^n:=u_h^n-\Pi u^n$ and $\eta^n:=\Pi u^n-u^n$. The estimate of $\eta^n$ is given by (\ref{ritz}) so we remain to bound $\xi^n$. We subtract (\ref{FEM}) from (\ref{ref}) with $\chi=\xi^n$ and  apply (\ref{Ritz}) and (\ref{r3r4})	into the resulting equation
	to obtain the following error equation in terms of $\xi^n$
	\begin{equation}\label{conv:e1}
	\begin{array}{rl}
	\ds (\delta_{\tau_n} \xi^n,\xi^n)+\big(\bm K\nabla \xi^n,\nabla \xi^n\big)&\ds= -k(t_n)(\delta_{\tau}^{1-\alpha(t_n)} \xi^n,\xi^n)\\[0.1in]
	&\ds- \big(k(t_n)(r_{2,n}-r_{4,n})+r_{1,n}+r_{3,n},\xi^n\big).
	\end{array}
	\end{equation}	
	We rearrange $\delta_{\tau}^{1-\alpha(t_n)}$ by
	$$\ds \delta_{\tau}^{1-\alpha(t_n)} \xi^n = \frac{1}{\Gamma(1+\alpha(t_n))} \Big [b_n^n \xi^n - \sum_{k=1}^{n-1} \big ( b^n_{k+1}-b_{k}^n\big) \xi^k - b_{1}^n \xi^0 \Big ]$$ 
	and apply $u_h^0:=\Pi u^0$ to reformulate (\ref{conv:e1}) as
	$$\begin{array}{l}
	\ds ( \xi^n,\xi^n)+\tau_n\big(\bm K\nabla \xi^n,\nabla \xi^n\big)+\frac{\tau_nk(t_n)b^n_n}{\Gamma(1+\alpha(t_n))} ( \xi^n,\xi^n)\\[0.1in]
	\ds\qquad=( \xi^{n-1},\xi^n) +\frac{\tau_nk(t_n)}{\Gamma(1+\alpha(t_n))}\sum_{k=1}^{n-1} \big ( b^n_{k+1}-b_{k}^n\big) (\xi^k,\xi^n)\\[0.1in]
	\ds\qquad\quad-\tau_n \big(k(t_n)(r_{2,n}-r_{4,n})+r_{1,n}+r_{3,n},\xi^n\big),
	\end{array}$$
	from which we use (\ref{bn}) to obtain
		\begin{equation}\label{conv:e2}
		\ds (1+A_n\tau_n^{\alpha(t_n)})\|\xi^n\|\leq \|\xi^{n-1}\|+A_n\tau_n  \sum_{k=1}^{n-1} \big (b_{k+1}^n - b^n_{k}\big) \|\xi^{k}\|+Q\tau_n\sum_{i=1}^4 \|r_{i,n}\|,
		\end{equation}
	where $A_n:=k(t_n)/\Gamma(1+\alpha(t_n))$. 
	
	We turn to evaluate the truncation error terms on the right-hand side of (\ref{conv:e2}). In the case $\alpha(0)=1$, $\alpha'(0)=0$ and $\lim_{t\rightarrow 0^+}\alpha'(t)\ln t$ is finite, the uniform partition is applied and by (\ref{thmLocal:e1}) and (\ref{thmtrun:e1}) we directly obtain 
	\begin{equation*}
	Q\tau\sum_{i=1}^4 \|r_{i,n}\|\leq QQ_0\tau \big(N^{-1}+h^2\big),
	\end{equation*}
with $Q_0$ defined in (\ref{thmLocal:e1}). Otherwise, the graded mesh with $r=2/\alpha(0)$ is chosen. We use (\ref{thmLocal:e2}), (\ref{thmtrun:e2}) and the fact that the bound of $r_{3,n}$ dominates that of $r_{4,n}$ (see Theorem \ref{thmtrun}) to obtain
	$$\begin{array}{rl}
	\ds Q\tau_n ( \|r_{2,n}\|+\|r_{1,n}\|)  & \ds \le \frac{QQ_0n^{r-1}}{N^r} \bigg \{ \frac{(1-\delta_{n,1})}{n}  \Big (\frac{N}{n} \Big)^{r(1-\alpha(0) - \alpha(t_n))} \\[0.1in]
	&\ds \qquad +  \delta_{n,1}N^{r(1-\alpha(0)-\alpha(t_1))} + \f1{n} \Big(\f{N}{n}\Big)^{r(1-\alpha(0))} \bigg\}\\[0.1in] 
	& \ds \hspace{-0.4in}\le QQ_0\Big(\frac{(1-\delta_{n,1}) n^{r\alpha(t_n)}}{N^{2+r\alpha(t_n)}} + \f{ \delta_{n,1}n^{r-1}}{N^{2 + r\alpha(1)}} +  \f{1}{N^2}\Big) \le \frac{QQ_0}{N^2}\leq \frac{QQ_0\tau_N}{N},
	\end{array}$$ 
	and
	\begin{equation*}
	\begin{array}{l}
	\ds  Q\tau_n ( \|r_{3,n}\|+\|r_{4,n}\|) \leq QQ_0h^2\frac{n^{r-1}}{N^r}\bigg(\frac{\delta_{n,1}+(1-\delta_{n,1})n^{r(\alpha(0)-1)}}{N^{r(\alpha(0)-1)}}\bigg)\\[0.15in]
	\hspace{0.2in}\ds\leq QQ_0h^2\big(\delta_{n,1}n^{r-1}+(1-\delta_{n,1})n^{r\alpha(0)-1}\big)N^{-r\alpha(0)}\leq QQ_0h^2N^{-1}\leq QQ_0h^2\tau_N.
	\end{array}
	\end{equation*}
	Therefore, in any case of $\alpha(t)$, we obtain the following estimates of the truncation errors under the appropriate temporal partition
	\begin{equation}\label{trunbnd}
	Q\tau_n\sum_{i=1}^4 \|r_{i,n}\|\leq Q'Q_0\tau_N \big(N^{-1}+h^2\big),
	\end{equation}
	 for some fixed constant $Q'$. We then prove the convergence estimates by mathematical induction.
Applying (\ref{trunbnd}) to (\ref{conv:e2}) with $n=1$ yields
	\begin{equation*}
    \|\xi^1\|\leq Q'Q_0\tau_N \big(N^{-1}+h^2\big).
	\end{equation*}
 Assume 
	\begin{equation}\label{induc:e1}
	\ds \|\xi^{m}\|\leq m\,Q'Q_0\tau_N \big(N^{-1}+h^2\big),~~2\leq m\leq n-1.
	\end{equation} 
	Plugging (\ref{trunbnd}) and (\ref{induc:e1}) with $2\leq m\leq n-1$ into (\ref{conv:e2}) leads to
		\begin{equation*}
		\begin{array}{rl}
		\ds (1+A_n\tau_n^{\alpha(t_n)})\|\xi^n\|&\ds\leq \|\xi^{n-1}\|+A_n\tau_n  \sum_{k=1}^{n-1} \big (b_{k+1}^n - b^n_{k}\big) \|\xi^{k}\|+Q\tau_n\sum_{i=1}^4 \|r_{i,n}\|\\[0.05in]
		&\ds\leq \Big[(n-1)\Big(1+A_n\tau_n  \sum_{k=1}^{n-1} \big (b_{k+1}^n - b^n_{k}\big) \Big)+1\Big]Q'Q_0\tau_N \big(N^{-1}+h^2\big)\\[0.05in]
		&\ds\leq [(n-1)(1+A_n\tau_n^{\alpha(t_n)})+1]Q'Q_0\tau_N \big(N^{-1}+h^2\big),
		\end{array}
		\end{equation*}	
in which we divide $1+A_n\tau_n^{\alpha(t_n)}$ on both sides to obtain (\ref{induc:e1}) for $m=n$ and thus for any $m\geq 2$ by mathematical induction. Then the proof is finished by applying $m\tau_N\leq rTm/N\leq rT$ into (\ref{induc:e1}). 	
\end{proof}

\section{Numerical experiments} 
We substantiate the analysis numerically by investigating the impact of $\alpha(t)$ on the convergence rate of the fully discrete finite element scheme (\ref{FEM}). 

Let $(\bm x,t)\in (0,1)^3\times[0,1]$, $k(t)=1$, $\bm K=\text{diag}(0.001,0.001,0.001)$ and
$$u=t^{\alpha(t)}\sin(2\pi x)\sin(2\pi y)\sin(2\pi z)$$
with
 $$\alpha(t) = \alpha(1) + (\alpha(0) - \alpha(1))\Big((1-t) - \frac{\sin(2\pi (1-t))}{2\pi} \Big),$$
	which satisfies $\alpha'(0)=0$ and $\lim_{t\rightarrow 0^+}\alpha'(t)\ln t=0$ and $f$ evaluated accordingly. We measure the convergence rates $\kappa$ and $\gamma$ such that
	\begin{equation*}
	\| u-u_h \|_{\hat L_\infty(0,T;L_2(\Omega))} \leq Q (N^{-\kappa}+h^{\gamma}).
	\end{equation*}
	 We select the uniform partition on space domain and the uniform temporal mesh is used for the case of $\alpha(0)=1$ (i.e., $u_{tt}$ is continuous on $[0,T]$). For the case of $\alpha(0)<1$ (i.e., the solutions exhibit singularity at $t=0$), both uniform and graded meshes with $r=2/\alpha(0)$ for time are applied. 
	  We present results in Table \ref{table:0} and \ref{table:1}, which reveal that the scheme (\ref{FEM}) with a uniform mesh has an optimal-order convergence rate for the case of smooth solutions (i.e., $\alpha(0)=1$), but only a sub-optimal order for the case of $\alpha(0)<1$. Instead, the scheme (\ref{FEM}) with the temporal graded mesh of $r = 2/\alpha(0)$ achieves an optimal-order convergence rate. These results coincide with Theorem \ref{thm:conv}.
	
	\begin{table}[H]
		\setlength{\abovecaptionskip}{0pt}
		\centering
		\caption{Convergence rates under $\alpha(1)=0.4$ and $\alpha(0)=1$ or $0.6$.}
		\vspace{0.5em}	
		{\footnotesize \begin{tabular}{ccccccc|ccc}
			\hline
			&Uniform&&Graded&&Uniform&&&Uniform&\\
			\cline{1-10}
			&$\alpha(0)=0.6$&  &$\alpha(0)=0.6$&&$\alpha(0)=1$ &&&$\alpha(0)=1$& \\
			\cline{1-10}
			$N$&$h=1/32$& $\kappa$ &$h=1/32$&$\kappa$&$h=1/32$ &$\kappa$&$h$&$N=1/h^2$&$\gamma$ \\
			\cline{1-10}		
			1/8&    3.26E-02&	    &	3.54E-02&	    &	1.00E-02&	    & 	1/8 & 	1.60E-03&      \\	
			1/16&	2.09E-02&	0.65&	1.87E-02&	0.92&	4.90E-03&	1.03&	1/16&	3.97E-04&	2.01\\
			1/32&	1.34E-02&	0.63&	9.60E-03&	0.97&	2.41E-03&	1.03&	1/24&	1.76E-04&	2.01\\
			1/64&	8.74E-03&	0.62&	4.85E-03&	0.98&	1.20E-03&	1.01&	1/32&	9.89E-05&	2.00\\
			
			\hline		
			\end{tabular}}
			\label{table:0}
			\end{table}
			
			\begin{table}[H]
				\setlength{\abovecaptionskip}{0pt}
				\centering
				\caption{Convergence rates under $\alpha(1)=0.6$ and $\alpha(0)=1$ or $0.8$.}
				\vspace{0.5em}	
				{\footnotesize \begin{tabular}{ccccccc|ccc}
					\hline
					&Uniform&&Graded&&Uniform&&&Uniform&\\
					\cline{1-10}
					&$\alpha(0)=0.8$&  &$\alpha(0)=0.8$&&$\alpha(0)=1$ &&&$\alpha(0)=1$& \\
					\cline{1-10}
					$N$&$h=1/32$& $\kappa$ &$h=1/32$&$\kappa$&$h=1/32$ &$\kappa$&$h$&$N=1/h^2$&$\gamma$ \\
					\cline{1-10}		
					1/8&    1.37E-02&		&    1.72E-02&	&    	6.65E-03&		&    1/8&   1.11E-03&\\	
					1/16&	7.98E-03&	0.78&	9.03E-03&	0.93&	3.34E-03&	0.99&	1/16&	2.76E-04&	2.00\\
					1/32&	4.65E-03&	0.78&	4.63E-03&	0.96&	1.68E-03&	0.99&	1/24&	1.22E-04&	2.00\\
					1/64&	2.72E-03&	0.77&	2.36E-03&	0.97&	8.60E-04&	0.97&	1/32&	6.89E-05&	2.00\\
					\hline		
				\end{tabular} }
				\label{table:1}
			\end{table}

\section{Appendix: Wellposedness of the variable-order tFDE and regularity of its solutions}
It is known \cite{CouHil,Eva} that the eigenfunctions $\{\phi_i \}_{i=1}^\infty$ of the Sturm-Liouville problem 
\begin{equation*}
\mathcal{L}\phi(\bm x) = \lambda_i \phi_i(\xx), ~\xx \in \Omega; \quad \phi_i(\xx) = 0, ~\xx \in \p \om
\end{equation*}
form an orthonormal basis in $L_2(\om)$. The eigenvalues $\{\lambda_i\}_{i=1}^\infty$ are positive and form a nondecreasing sequence that tend to $\infty$ with $i$. We use the theory of sectorial operators to define the fractional Sobolev spaces \cite{SakYam,Tho} 
$$\begin{array}{c}
\ds \check{H}^{\gamma}(\om) := \Big \{v \in L_2(\om): | v |_{\check{H}^\gamma}^2 :=  \sum_{i=1}^{\infty} \lambda_i^{\gamma} (v,\phi_i)^2 < \infty \Big \},
\end{array}$$  
with the norm being defined by $\| v \|_{\check{H}^\gamma} := \big ( \|v\|^2 + | v |_{\check{H}^\gamma}^2 \big)^{1/2}$. Note that $\check{H}^\gamma(\om)$
is a subspace of the fractional Sobolev space $H^\gamma(\om)$ characterized by \cite{AdaFou,SakYam,Tho}
$$
\check{H}^\gamma(\om) = \big \{v \in H^\gamma(\Omega): ~\mathcal L^s v(\xx) = 0, ~\xx \in \p \om,~~ s < \gamma/2  \big \}
$$
and the seminorms $| v |_{\check{H}^\gamma}$ and $| v |_{{H}^\gamma}$ are equivalent in $\check{H}^\gamma$.

Then using the approaches of variable seperation, the wellposedness of the model (\ref{ModelR}) and the regularity estimates of its solutions $u(x)$ are proved by the following theorems \cite{WanZheJMAA}.
\begin{theorem}\label{thm:Wellpose}
	Suppose that $u_0 \in \check{H}^{\gamma + 2}$, $f\in H^1(0,T;\check{H}^\gamma)$ for some $d/2 < \gamma \in \mathbb{R}^+$ and Assumption A holds. If $ \alpha(0) = 1$ then problem (\ref{ModelR}) has a unique solution $u \in C^1\big([0,T];\check H^\gamma \big)$ with the following stability estimates for any $0\leq s \leq\gamma$
	\begin{equation*}\begin{array}{rl}
	\|u\|_{C([0,T];\check H^s)} &\ds  \leq Q \big(\|u_0\|_{\check H^s}+\|f\|_{L_2(0,T;\check H^s)}\big), \\[0.075in]
	\|u\|_{C^1([0,T];\check H^s)} & \ds \le Q \big(\|u_0\|_{\check H^{2+s}} + \|f\|_{H^1(0,T;\check H^s)}\big).
	\end{array}\end{equation*}
	If $\alpha(0) < 1$ problem (\ref{ModelR}) has a unique solution $u \in C\big([0,T];\check H^\gamma \big) \cap C^1\big((0,T];\check H^\gamma \big) $ with the stability estimate 
	\begin{equation}\label{Wellpose:e2}\begin{array}{rl}
	\hspace{-0.125in} \|u\|_{C([0,T];\check H^s)} & \ds  \le Q \big(\|u_0\|_{\check H^{2+s}} + \|f\|_{H^1(0,T;\check H^s)}\big), \\[0.1in]
	\hspace{-0.125in} \|u\|_{C^1([\varepsilon,T];\check H^s)} & \ds  \le Q \varepsilon^{\alpha(0)-1} \big(\|u_0\|_{\check H^{2+s}} + \|f\|_{H^1(0,T;\check H^s)}\big),~~0 < \varepsilon \ll 1.
	\end{array}\end{equation}
    Here $Q = Q\big(\alpha_m, \| k\|_{C[0,T]},T\big)$.
\end{theorem}		

\begin{theorem}\label{thm:u2}
	Suppose that $u_0\in \check{H}^{s+4}$, $f \in H^1(0,T;\check{H}^{s+2})\cap H^2(0,T;\check{H}^s)$ for $s\geq 0$, and that $\alpha, k \in C^1[0,T]$ and (\ref{alpha}) holds. Then the following conclusions hold:
	\paragraph{Case 1.} If $\alpha(0) = 1$, $\alpha'(0)=0$ and $\lim_{t\rightarrow 0^+}\alpha'(t)\ln t$ is finite, then $u\in C^2([0,T];$ $\check H^s)$ and 
	\begin{equation*}
	\|u\|_{C^2([0,T];\check H^s)}\leq Q \big(\|u_0\|_{\check{H}^{s+4}}+\|f\|_{H^1(0,T;\check{H}^{s+2})}+\|f\|_{H^2(0,T;\check{H}^s)}\big).
	\end{equation*}
	\paragraph{Case 2.} If $\alpha(0) = 1$ but $\alpha'(0)\neq 0$ or $\lim_{t\rightarrow 0^+}\alpha'(t)\ln t$ is not finite, then $u\in C^1([0,T];\check H^s)\cap C^2((0,T];\check H^s)$ and for any $0 < \varepsilon\ll 1$
	\begin{equation*}
	\|u\|_{C^2([\varepsilon,T];\check H^s)}\leq K|\ln\varepsilon| \big(\|u_0\|_{\check{H}^{s+4}}+\|f\|_{H^1(0,T;\check{H}^{s+2})}+\|f\|_{H^2(0,T;\check{H}^s)}\big).
	\end{equation*}
	\paragraph{Case 3.} If $\alpha(0) < 1$, then $u\in C^2((0,T];\check H^s(\Omega))$ and for any $0 < \varepsilon\ll 1$
	\begin{equation*}
	\|u\|_{C^2([\varepsilon,T];\check H^s)}\leq Q\varepsilon^{\alpha(0)-2} \big(\|u_0\|_{\check{H}^{s+4}}+\|f\|_{H^1(0,T;\check{H}^{s+2})}+\|f\|_{H^2([0,T];\check{H}^s)}\big).
	\end{equation*}
	Here $Q = Q\big(\alpha_m, \| k\|_{C^1[0,T]},T\big)$.

\end{theorem}
%

\section*{Acknowledgements}

This work was funded by the OSD/ARO MURI Grant W911NF-15-1-0562 and the National Science Foundation under Grant DMS-1620194.

\end{document}